\documentclass[a4paper,11pt]{article}
\usepackage{graphicx}
\usepackage[T1]{fontenc}
\usepackage[utf8]{inputenc}
\usepackage[english]{babel}
\usepackage{amsmath,amsthm,amssymb,amsfonts}
\usepackage{lmodern}

\usepackage{geometry}
 \usepackage{color}
\definecolor{darkgreen}{rgb}{0.00,0.5,0.00}
\usepackage{hyperref}
\hypersetup{
	colorlinks=true,        
	linkcolor=blue,         
	citecolor=darkgreen,         
	urlcolor=blue           
}

\usepackage{cite}
\usepackage{enumerate}
\usepackage{empheq}
\mathtoolsset{showonlyrefs}

\newtheorem{theorem}{Theorem}
\theoremstyle{definition}
\newtheorem{lemma}{Lemma}

\newtheorem{remark}{Remark}

\renewcommand{\a}{\alpha}
\renewcommand{\b}{\beta}
\renewcommand{\c}{\gamma}
\renewcommand{\d}{\delta}

\newcommand{\la}{\lambda}
\renewcommand{\t}{\tau}
\newcommand{\s}{\sigma}
\newcommand{\e}{\varepsilon}
\renewcommand{\th}{\theta}

\newcommand{\x}{\bar x}
\newcommand{\hx}{\hat x}

\newcommand{\R}{\mathbb R}

\newcommand{\N}{\mathbb N}

\newcommand{\Pp}{\mathcal P}

\newcommand{\n}[1]{\|#1 \|}
\newcommand{\lr}[1]{\langle #1\rangle}
\newcommand{\T}{\top}

\DeclareMathOperator{\prox}{prox}
\DeclareMathOperator*{\argmin}{argmin}
\DeclareMathOperator{\dom}{dom}

\DeclareMathOperator{\ri}{ri}

\title{The primal-dual hybrid gradient method reduces to a primal
    method for linearly constrained optimization problems\thanks{The
        research was supported by the German Research Foundation grant
        SFB755-A4}}

\author{Yura Malitsky\footnote{Institute for Numerical and Applied
        Mathematics, University of G\"ottingen,
        \href{mailto:y.malitsky@gmail.com}{y.malitsky@gmail.com}}}

\date{}

\begin{document}
\maketitle

\begin{abstract}
    In this work, we show that for linearly constrained optimization
    problems the primal-dual hybrid gradient algorithm, analyzed by
    Chambolle and Pock~\cite{chambolle2011first}, can be written as an
    entirely primal algorithm.  This allows us to prove convergence of
    the iterates even in the degenerate cases when the linear system
    is inconsistent or when the strong duality does not hold.  We also
    obtain new convergence rates which seem to improve existing ones
    in the literature. For a decentralized distributed optimization we
    show that the new scheme is much more efficient than the original
    one.
\end{abstract} {\small \textbf{Keywords. } First-order algorithms
    $\cdot$ primal-dual algorithms $\cdot$ convergence rates $\cdot$
    linearly constrained optimization problem $\cdot$ penalty methods
    $\cdot$ distributed optimization
    \medskip }

\section{Introduction}
In this paper, we study nonsmooth optimization problems with linear
constraints of the form
\begin{equation}\label{lin_con}
  \min_{x\in \R^{n}} \quad g(x)   \quad \text{s.t.} \quad Ax = b,
\end{equation}
where $A\in \R^{m\times n}$, $b\in \R^m$, and
$g\colon \R^n \to (-\infty,+\infty]$ is a proper lower semicontinuous
convex function. Problem~\eqref{lin_con} is one of the most importance
in convex optimization and, in particular, includes conic
optimization, which in turn includes linear and semidefinite
optimization.  Problem~\eqref{lin_con} often arises in machine
learning, inverse problems, and distributed optimization. Notice that
every composite optimization problem $\min_u g_1(u) + g_2(Ku)$ can be
also written in the form~\eqref{lin_con} by setting $x = (u,v)^{\T}$,
$g(x) = g_1(u)+g_2(v)$,
$A = [\begin{array}{c|c} K^{\T} & -I \end{array}]^{\T}$, and $b = 0$.
In the sequel, we restrict our attention to large-scale problems where
computing the projection onto the subspace $\{x\colon Ax = b\}$ is
expensive or even practically impossible.

In this work, we focus on the primal-dual hybrid algorithm (PDHG),
analyzed by Chambolle and Pock in \cite{chambolle2011first}. It is a
popular method for solving convex-concave saddle problems with a
bilinear term, owing to the fact that its iteration requires computing
only two proximal operators and two matrix-vector
multiplications. This is different from the alternating direction
method of multipliers (ADMM)---another popular approach to solve
\eqref{lin_con}, where each iteration requires solving a nontrivial
problem. However, on the other hand, the PDHG algorithm is a
particular case of a more general proximal
ADMM~\cite{banert2016fixing,shefi2014rate}. For possible extensions
and applications of the PDHG, we refer the reader to
\cite{chambolle2016introduction,malitsky2016first,Condat2013,vu2013splitting,he-yuan:2012}.

By introducing the Lagrange multiplier $y$, one can rewrite \eqref{lin_con}  as a saddle point problem
\begin{equation}\label{lin_con_saddle}
  \min_{x\in \R^{n}} \max_{y\in \R^m}\quad g(x) + \lr{Ax,y} - \lr{b,y}.
\end{equation}
Then the PDHG applied to \eqref{lin_con_saddle} generates sequences
$(x^k)$, $(y^k)$ according to
  \begin{equation}\label{eq:1}
        \begin{cases}
        y^{k+1}  =  y^k + \sigma (A \x^k -b)\\
        x^{k+1}  = \prox_{\t g}(x^k - \tau A^{\T} y^{k+1}),
    \end{cases}
\end{equation}
where $x^0\in \R^n, y^0\in \R^m$ are arbitrary and $ \x^{k} = 2x^{k} -x^{k-1}$ for all
$k\in \N$. The convergence result stated in \cite{chambolle2011first} says
that if there exists a saddle point for problem~\eqref{lin_con_saddle} and
$\t \s \n{A}^2<1$, then $(x^k, y^k)$ converges to a saddle point of
\eqref{lin_con_saddle}. Moreover, in this case, we have $O(1/k)$
ergodic rate for the primal-dual gap.

In this work, we are interested in cases which are not covered by the above
statement. More precisely, we provide answers to the following two questions.
\begin{itemize}
    \item \emph{What will happen to the iterates of \eqref{eq:1} if problem
    \eqref{lin_con} is infeasible?}
  \item \emph{What will happen to the iterates of \eqref{eq:1} if there is no saddle
    point in \eqref{lin_con_saddle}?}
\end{itemize}

Note that the standard analysis of the PDHG in
\cite{chambolle2011first}, or alternative ones in
\cite{he-yuan:2012,malitsky2016first}, cannot resolve aforementioned
issues.  To answer these questions, we show that the primal-dual
algorithm \eqref{eq:1} can be reformulated as an entirely primal
algorithm, \emph{i.e.,} without resorting to dual variables. This is
in fact our main result, as it views the primal-dual algorithm
\eqref{eq:1} from a new perspective. We reveal a connection of
\eqref{eq:1} to the accelerated proximal gradient
method~\cite{tseng08} for the composite minimization.  A novel
analysis yields new convergence rates which are more suitable in
practice and seem to be better than existing ones. We show a
connection of the PDHG to the diagonal penalty methods and inverse
problems. The new scheme is simpler than the original one to implement
and has a smaller memory footprint. In fact, at least from an
algorithmic point of view, it might be the simplest existing scheme
for solving such a generic problem \eqref{lin_con}.  In contrast to
the standard PDHG method, it can be applied to a decentralized
distributed optimization problem with only one communication per
iteration. Moreover, it also achieves a better complexity. The new
scheme is favorable to new extensions. For example, in the subsequent
paper \cite{luke2018block}, inspired by the coordinate extension of
Tseng's method~\cite{tseng08} proposed by Fercoq and
Richt\'arik~\cite{fercoq2015accelerated}, we derive a coordinate
extension of the PDHG for \eqref{lin_con}.

\smallbreak

\textbf{Paper outline}. In section~\ref{sec:pre} we briefly recall the
standard notation from convex analysis and establish several
preparatory lemmas. Section~\ref{sec:lin_con} is dedicated to the new
analysis of the PDHG algorithm \eqref{eq:1} and its consequences. In
section~\ref{sec:gen} we consider several generalizations of the PDHG
method: when $g$ is strongly convex and when the objective in
\eqref{lin_con} has an additional smooth term.

\section{Preliminaries}\label{sec:pre}
Throughout the paper we will work in a finite-dimensional vector space
$\R^n$ equipped with an inner product $\lr{\cdot,\cdot}$ and a norm
$\n{\cdot} = \sqrt{\lr{\cdot,\cdot}}$.
A function $g$ is called $\c$--strongly convex function, if
$g - \frac{\c}{2}\n{\cdot}^2$ is convex. For a convex lower
semi-continuous (lsc) function $g\colon \R^n\to (-\infty, +\infty]$ we
denote by $\dom g$ the domain of $g$, \emph{i.e.}, the set
$\{x \colon g(x)< + \infty\}$, and by $\prox_g$ the proximal operator
of $g$ that is $\prox_g(z) = \argmin_x\{g(x)+\frac 1
2\n{x-z}^2\}$. The following characteristic property (prox-inequality)
will often be used:
\begin{equation}\label{prox_charact}
    \x = \prox_{g}z \quad \Leftrightarrow \quad \lr{\x - z, x- \x} \geq g(\x)
    - g(x) \quad \forall x\in \R^n.
\end{equation}
When $g$ is $\c$--strongly convex, the above inequality can be strengthened:
\begin{equation}\label{prox_charact-strong}
    \x = \prox_{g}z \quad \Leftrightarrow \quad \lr{\x - z, x- \x} \geq g(\x)
    - g(x)+\frac{\c}{2}\n{\x - x}^2 \quad \forall x\in \R^n.
\end{equation}

For a linear operator $A\colon \R^{n}\to \R^m$ with a slight abuse of
notation we denote its operator norm as $\n{A}$. Throughout the paper
$f(x)=\frac 1 2\n{Ax-b}^2$, where $x\in \R^n, b\in \R^m$, and
$f_* := \min_x f(x)$.  Since $f$ is a quadratic, we have
\begin{equation}
    \label{eq:quad-id1}
    \a f(x) + (1-\a) f(y) =  f\bigl(\a x + (1-\a)y\bigr) +
    \frac{\a(1-\a)}{2}\n{A(x-y)}^2\quad \forall x,y\in \R^n\ \forall
    \a \in \R.
\end{equation}
\begin{equation}
    \label{eq:quad-id2}
    f(x) - f_* = f(x) - f(\x) = \frac 1 2 \n{A(x-\x)}^2 \quad \forall x\in \R^n \,\, \forall
    \x \in \argmin f.
\end{equation}
Another useful identity (the cosine law) obviously holds:
\begin{equation}
    \label{eq:cosine}
    2\lr{x-y,z-x} = \n{y-z}^2 - \n{x-y}^2 -\n{x-z}^2 \qquad \forall
    x,y,z \in \R^n.
\end{equation}
We conclude our preliminary section by two important lemmas.
\begin{lemma}\label{fejer}
    Suppose that sequences $(x^k)\subset \R^n$, $(b_k)\subset \R $ and
    a set $D\subset \R^n$ satisfy:
    \begin{enumerate}[(i)]
        \item all cluster points of $(x^k)$ belong to $D$;

        \item for all $x\in D$ the sequence $(\n{x^k-x}^2+b_k)$ is
        nonincreasing and bounded  below.
    \end{enumerate}
    Then the sequence $(x^k)$ converges to some point in $D$.
\end{lemma}
\begin{proof}
    Suppose, on the contrary, that there exist two different
    subsequences $(x^{k_i})$ and $(x^{k_j})$ such that
    $x^{k_i}\to \tilde x_1$, $x^{k_j}\to \tilde x_2$ and
    $\tilde x_1 \neq \tilde x_2$. Let
    $a_k(x):= \n{x^k-x}^2 + b_k$.  By (ii), the sequence
    $(a_k(x))$ is convergent for any $x\in D$. Setting
    $x:=\tilde x_1$, we obtain
\begin{align*}
\lim_{k\to \infty }a_k(\tilde x_1) & = \lim_{i\to \infty}a_{k_i}(\tilde x_1) =
\lim_{i\to\infty}(\n{x^{k_i}-\tilde x_1}^2 + b_{k_i})
=\lim_{i\to\infty}b_{k_i}\\ & = \lim_{j\to\infty}a_{k_j}(\tilde x_1) =
\lim_{j\to\infty}(\n{x^{k_j}-\tilde x_1}^2 + b_{k_j}) =
\n{\tilde x_2-\tilde x_1}^2 + \lim_{j\to\infty}b_{k_j},
\end{align*}
from which $\lim_{i\to\infty}b_{k_i} = \n{\tilde x_2-\tilde x_1}^2 +
\lim_{j\to\infty}b_{k_j}$ follows. Setting $x=\tilde x_2$, we analogously derive
$$\lim_{j\to\infty}b_{k_j} = \n{\tilde x_1-\tilde x_2}^2
+ \lim_{i\to\infty}b_{k_i},$$
from which we conclude that $\tilde x_1 = \tilde x_2$. Therefore, the
whole sequence $(x^k)$ converges to some point in $D$.
\end{proof}

\begin{lemma} \label{l:3points} Let $f(x) = \frac 1 2 \n{Ax-b}^2$,
    $f_* = \min_x f(x)$ and $\x \in \argmin f$. Then for any
    $u,v \in \R^n$ it holds
\begin{equation*}
\lr{\nabla f(u), \x-v} =2f_* -f(u)-f(v) + \frac 1 2 \n{A(u-v)}^2.
\end{equation*}
\end{lemma}
\begin{proof}
    As $\x\in \argmin f$, $A^{\T}A\x = A^{\T}b$. Using this, we have
    \begin{align*}
        \lr{\nabla f(u), \x-v} = \lr{A^{\T}(Au-b), \x-v} & =  \lr{A^{\T}A(u-\x),
        \x-v} =  \lr{A(u-\x), A(\x-v)} \\ & = \frac{1}{2} \n{A(u-v)^2} - \frac
        1 2\n{A(u-\x)}^2 - \frac 1 2 \n{A(v-\x)}^2.
    \end{align*}
    Then the statement follows directly from \eqref{eq:quad-id2}.
\end{proof}



\section{Main part} \label{sec:lin_con} In this section, we describe
several cases where known analyses of the PDHG are non applicable.

\textbf{Infeasible problem.} The first issue arises when the
constraints in \eqref{lin_con} are inconsistent:
$\mathcal{P} := \{ x \colon Ax = b \} = \varnothing$.
Clearly, in this
case problem \eqref{lin_con}, as a minimization problem over an empty
set, does not have a lot of sense.  If we know in advance that
$\Pp = \varnothing$, probably the most natural thing is to consider
the following generalization of~\eqref{lin_con}:
\begin{equation}\label{lin_con_gen}
  \min_{x\in \R^n} \quad g(x)   \quad \text{s.t.} \quad x\in
  \argmin_{u\in \R^n} f(u),
\end{equation}
where $f(x) = \frac 1 2 \n{Ax-b}^2$. Now the constraints are
always nonempty. Moreover, by optimality condition for $f$, one can equivalently recast
\eqref{lin_con_gen} as
\begin{equation}\label{lin_con_diff}
  \min_{x\in \R^n} \quad g(x)   \quad \text{s.t.} \quad A^{\T}Ax = A^{\T}b.
\end{equation}
Hence, we again have a problem of the same class as \eqref{lin_con}
but now it is always feasible. Thus, we still can apply the PDHG to
\eqref{lin_con_diff}, which yields the following recursion:
\begin{equation}\label{pd:lin2}
    \begin{cases}
        y^{k+1}  = y^k + \sigma (A^{\T}A \x^k -A^{\T}b)\\
        x^{k+1}  = \prox_{\t g}(x^k - \tau A^{\T}A y^{k+1}),
      \end{cases}
    \end{equation}
    where $\t \s \n{A}^4 < 1$ and, as before,
    $\x^{k} = 2x^{k} - x^{k-1}$. This scheme has several
    drawbacks, compared to \eqref{eq:1}. First, it uses four
    matrix-vector multiplications. Of course, one can precompute
    $A^{\T}A$, but for large-scaled problems it might be expensive and
    often not desirable. For instance, if $A$ is sparse, $A^{\T}A$ will
    probably become denser, which will lead to a more expensive
    iteration. Third, it is known that the conditional number of
    $A^{\T}A$ is a square of the conditional number of $A$ \cite[Chapter
    8]{saad}. Hence, when $A$ is ill-conditioned, working with $A^{\T}A$
    will be much harder than with $A$. Finally, the stepsizes in
    \eqref{pd:lin2} have to satisfy a more restrictive inequality
    $\t \s \n{A}^4< 1$. Because of that, \eqref{pd:lin2} will have
    worse estimates for the convergence rate.

Another reason not to apply algorithm \eqref{pd:lin2} to solve
\eqref{lin_con_gen} is the absence of a priori knowledge that the set
$\Pp$ is empty. Thus, the best situation would be to have something
meaningful from the iterates of \eqref{eq:1}. We show that this is
indeed the case: the PDHG algorithm given by \eqref{eq:1} solves a
general problem \eqref{lin_con_gen}, so there is no need to apply
more expensive scheme \eqref{pd:lin2}.

\textbf{No strong duality.} The convergence of almost all widely-used
methods (PDHG, ADMM, variational inequality methods) for
solving~\eqref{lin_con} heavily relies on the duality arguments. The
assumption that the strong duality holds is usually taken for granted,
although sometimes it is not so easy to check whether it is
satisfied. The standard condition that ensures that strong duality
holds for problem~\eqref{lin_con} is $b \in \ri (A\dom g)$, where
$\ri(C) $ stands for the relative interior of $C$. Whenever $g$ is not
full-domain, \emph{i.e.,}  not finite-valued, it becomes non-trivial
to verify it.

\subsection{Primal form of PDHG}
\label{sec:lin:appr}
Here we show that the primal-dual method applied to~\eqref{lin_con}
can be seen as a modified Tseng's method~\cite{tseng08}, whose stepsizes tend to
infinity.

Reducing the primal-dual algorithm \eqref{eq:1} to the only primal
form is in fact trivial.  Iterating the first equation in
\eqref{eq:1}, one can derive
\begin{align*}
       y^{k+1} & = y^k + \s (A\x^k-b) = y^{k-1} + \s A(\x^k+\x^{k-1})
       -2\s b = \dots \\ &  = y^0 + \s A(\x^k+\dots +
       \x^0) - (k+1)\s b.
   \end{align*}
   For simplicity, assume that the PDHG starts from $(x^0, y^0)$ with
   $\x^0 = x^0$ and $y^0 = 0$. Then the above equation is equivalent
   to $y^1 = \s (Ax^0-b)$ if $k=0$, to $y^2 = 2\s (Ax^1- b)$ if $k=1$,
   and to
\begin{equation}
    \label{lin_con:y}
    y^{k+1} = \s A (\x^k + \dots + \x^0) - \s (k+1) b =\s A(2x^k +
    x^{k-1}+\dots + x^1) - \s (k+1) b,
\end{equation}
if $k\geq 2$.  Define a new sequence $(z^k)$ with $z^0 = x^0$,
$z^1 = x^1$ and $z^k = \frac{1}{k+1}(2x^k + x^{k-1}+\dots + x^1)$ for
all $k\geq 2$. Then $y^{k+1} = (k+1)\s (Az^k- b)$ for all $k\geq 0$
and hence, the primal-dual scheme \eqref{eq:1} can be written as
\begin{equation}\label{lin:pd-1}
    \begin{cases}
  z^k  = \frac{k}{k+1}z^{k-1}+\frac{1}{k+1}\x^k\\
  x^{k+1}  = \prox_{\t g} (x^k - (k+1) \t \s A^{\T}(Az^k-b)),
    \end{cases}
\end{equation}
where $k\geq 0$. This is easy to see: the sequence $(z^k)$, defined as
above, obviously satisfies the first recurrent equation in
\eqref{lin:pd-1} and the second equation is a direct consequence of
one in \eqref{eq:1}. We continue transforming the scheme
\eqref{lin:pd-1} by introducing another sequence $(s^k)$ with
$s^0 = x^0$ and $s^k = \frac{x^1+\dots + x^k}{k}$ for $k\geq 1$. Then
\eqref{lin:pd-1} can be cast as
\begin{equation}    \label{pd2}
\begin{cases}
    z^k  = (x^k+ks^{k})/(k+1)\\
    x^{k+1} = \prox_{\t g}(x^k - (k+1)\la \nabla f(z^k))\\
     s^{k+1} = (x^{k+1}+ks^k)/(k+1),
\end{cases}
\end{equation}
where $k\geq 0$ and as usually $f(x) = \frac 1 2 \n{Ax-b}^2$ and
$\la = \tau \sigma $. Remember that the iterates $(x^k)$ in the scheme
\eqref{pd2} are exactly the same as in the PDHG method (with
$y^0=0$). In the theorem below we will show that these iterates
converge in fact to a solution of
\begin{equation}\label{lcg}
  \min_{x\in \R^n} \quad g(x)   \quad \text{s.t.} \quad x\in
  \argmin f,
\end{equation}
which is a more general problem than the original \eqref{lin_con}.

Based on \cite{lan2011primal,nesterov1983method,nesterov07,fista},
Tseng in \cite{tseng08} proposed a simple and elegant way to analyze
accelerated gradient methods of Nesterov for a problem of composite
minimization $\min_x g(x)+h(x)$, where $h\colon \R^n \to \R$ is a
convex smooth function. Among several schemes that Tseng proposed one was the following
\begin{equation}\label{eq:tseng}
  \begin{cases}
    z^k  = \th_k x^k + (1-\th_k) s^{k}\\
    x^{k+1}  = \prox_{\frac{\la}{\th_k} g}(x^k - \frac{\la}{\th_k} \nabla h(z^k))\\
    s^{k+1}  = \th_k x^{k+1} + (1-\th_k) s^k.
  \end{cases}
\end{equation}
Convergence of \eqref{eq:tseng} was proved under the assumption that
$\nabla h$ is $\la^{-1}$--Lipschitz continuous and $\th_k\in (0,1]$
satisfies $\frac{1-\th_k}{\th_k^2} \leq \frac{1}{\th_{k-1}^2}$. The
simplest choice for such $(\th_k)$ is $\th_k = \frac{2}{k+2}$.

It is easy to see how similar \eqref{eq:tseng} and \eqref{pd2}
are. What are the differences? First, the PDHG algorithm uses
$\th_k = \frac{1}{k+1}$, which does not satisfy condition for
$(\th_k)$ in \eqref{eq:tseng} and goes slightly faster to zero than
$\th_k = \frac{2}{k+2}$. This is only due to the fact that $f$ is a
quadratic function, for which we can use tighter estimates. In fact,
the same can be done for the Tseng algorithm in the case $h =
f$. Second, in every iteration Tseng's scheme \eqref{eq:tseng} uses
the same stepsizes for both $g$ and $h$ and this is natural, as both
these functions are independent in the composite minimization
problem. In contrast, in problem \eqref{lin_con_gen} $f$ and $g$ are
not equivalent: $f$ impose hard constraints, thus the stepsize for $f$
goes to infinity. Third, in \eqref{eq:tseng} $h$ can be an arbitrary
function (up to the restrictions above), while in \eqref{pd2} $f$ is a
quadratic function.  It is also interesting to remark that in the case
$h \equiv f$ and $g$ is the indicator function of some closed convex
set, both schemes \eqref{pd2} and \eqref{eq:tseng} coincide: all
equations are the same, only $(\th_k)$ will be slightly different, see
the discussion above.

In fact, since $\nabla f$ is linear, we do not need variable
$z^k$ in \eqref{pd2} at all. Evidently, the scheme~\eqref{pd2} can be cast in a simpler way as
\begin{equation}\label{eq:pd-simpler}
      \begin{cases}
        x^{k+1}  = \prox_{\t g}(x^k - \la \nabla f(x^k +  k
        s^k))\\
        s^{k+1} = (x^{k+1} + ks^k)/(k+1),
    \end{cases}
\end{equation}
where $k\geq 0$ and $x^0 = s^0$.

For convenience, we recall our assumptions and define some notations:
\begin{gather}
 g\colon \R^n\to (-\infty, +\infty] \quad \text{is convex lsc}, \quad  A\in
  \R^{m\times n},\, b\in  \R^m, \quad  f(x) = \frac 1 2 \n{Ax-b}^2,\\
  f_* = \min_x f(x), \quad  g_* = \min_{x\in \argmin f} g(x), \quad  F_k (x) = g(x) + \s k (f(x) - f_*).
\end{gather}
Now we can state our main result.
\begin{theorem}\label{th_lin_cons}
    Assume that the solution set $S$ of \eqref{lcg} is
    nonempty, $\t, \la >0$, and $\la \n{A}^2 <1$. Then for sequences
    $(x^k)$, $(s^k)$, generated by \eqref{pd2} (or \eqref{eq:pd-simpler}), it
    holds
\begin{enumerate}[(i)]
    \item $F_k(s^k) - g_*= O(1/k)$.
      \smallskip
      
        \item If strong duality holds for problem~\eqref{lcg}, then $(x^k)$ and $(s^k)$ converge to a
        solution of~\eqref{lcg}  and $f(x^k)-f_* = O(1/k)$,
        $f(s^k)-f_* = O(1/k^2)$,
        $|F_k(s^k)-g_*|=O(1/k)$, $|g(s^k)-g_*| = O(1/k)$.
        
\smallskip
       \item If $S$ is bounded and $g$ is bounded  below, then all
       cluster points of $(s^k)$ belong to $S$ and
       $f(s^k)-f_* = o(1/k)$.
    \end{enumerate}
\end{theorem}
Therefore, in the most general case one can consider $s^k$ as an
approximated minimizer of the problem $\min_x F_k(x)$. Later we will
show how this latter problem is related to penalty methods. When the
strong duality holds, it is possible to prove convergence of the
iterates and derive some important rates.  Finally, when there is no
strong duality, but the solution set $S$ is bounded and the function
$g$ is bounded below, we still can say something meaningful about
convergence of the iterates $(s^k)$ and the rate of the feasibility
gap $f(s^k)-f_*$. The latter conditions are usually easy to check in
advance, in contrast to the strong duality.

The strong duality plays such an important role here because it allows
us to obtain a key estimate to prove global convergence. Specifically,
assume that the strong duality holds for problem~\eqref{lcg} and
$(x^*,u^*)\in \R^n\times \R^n $ is a saddle point of
\begin{equation*}
    \min_x \max_u g(x) + \lr{u, A^{\T}Ax - A^{\T}b}.
\end{equation*}
This means that $0 \in A^{\T}A u^*+ \partial g(x^*)$ and hence,
\begin{equation}\label{exist_dual}
g(x)-g_*\geq \lr{-A^{\T}A u^*, x-x^*} \geq
-\n{Au^*}\cdot \n{A(x-x^*)} = - D_y \cdot \sqrt{2(f(x)-f_*)},
\end{equation}
where for simplicity $D_y :=  \n{Au^*}$. Notice that in
the consistent case a saddle point for \eqref{lin_con} is $(x^*, y^*)$
with $y^* = Au^*$ Thus, the above estimate recovers a more common
one $g(x)-g_*\geq -\n{y^*}\n{Ax-b} $ and therefore, in this case $D_y = \n{y^*}$.

\begin{proof}
    Let $\x\in S$. By the prox-inequality~\eqref{prox_charact} and
    linearity of $\nabla f$,
\[
\frac{1}{\la}\lr{x^{k+1}-x^k, \x-x^{k+1}} + \lr{\nabla f(x^k),\x-x^{k+1}} +
 k \lr{\nabla f(s^k), \x-x^{k+1}} \geq \frac 1 \s (g(x^{k+1}) -g_*).
\]
From Lemma~\ref{l:3points} it follows
\begin{equation}\label{identities}
\begin{aligned}
\lr{\nabla f(x^k),\x-x^{k+1}} & = 2f_* - f(x^{k+1})-f(x^k) + \frac 1 2
\n{A(x^{k+1}-x^k)}^2, \\
\lr{\nabla f(s^k),\x-x^{k+1}} & = 2f_* - f(x^{k+1})-f(s^k) + \frac 1 2
\n{A(x^{k+1}-s^k)}^2.
\end{aligned}
\end{equation}
Using these identities in the above inequality, we deduce
\begin{multline}\label{lin_alter-1}
(k+1)(f(x^{k+1})-f_*)  +(f(x^k)-f_*) + k(f(s^k)-f_*) + \frac{1}{\s}(g(x^{k+1}) -g_*) -\frac{1}{2}\n{A(x^{k+1}-x^k)}^2 \\ \leq \frac{1}{\la}\lr{x^{k+1}-x^k, \x-x^{k+1}}
 +\frac{k}{2}\n{A(x^{k+1}-s^k)}^2.
\end{multline}
Convexity of $F_{k+1}(x)= g(x)+ \s (k+1) (f(x)-f_*)$ and the
property~\eqref{eq:quad-id1} for $f$ yield
\begin{equation}
    \label{eq:f-quadr}
F_{k+1}(x^{k+1}) + kF_{k+1}(s^k) \geq (k+1)F_{k+1}(s^{k+1})
+\frac{\s k}{2}\n{A(x^{k+1}-s^k)}^2.
\end{equation}
Applying \eqref{eq:f-quadr} to \eqref{lin_alter-1} and using that
$kF_{k+1}(s^k)= kF_k(s^k) + \s k (f(s^k)-f_*)$, we obtain
\begin{multline}\label{lin_alter-2}
\frac{k+1}{\s}F_{k+1}(s^{k+1}) +(f(x^k)-f_*)
-\frac{1}{2}\n{A(x^{k+1}-x^k)}^2 - \frac{1}{\s} g_*  \leq \frac{1}{\la}\lr{x^{k+1}-x^k,
\x-x^{k+1}}+ \frac{k}{\s}F_k(s^k).
\end{multline}
Finally, using the cosine law \eqref{eq:cosine} and
$\n{A(x^{k+1}-x^k)}\leq \n{A}\n{x^{k+1}-x^k}$, we arrive at
\begin{align}\label{lin_alter-3}
\frac{1}{2\la}\n{x^{k+1}-\x}^2 &+\frac{k+1}{\s}(F_{k+1}(s^{k+1})-g_*) +
\frac{1-\la \n{A}^2}{2\la}\n{x^{k+1}-x^k}^2 \nonumber \\ &+(f(x^k)-f_*) \leq
\frac{1}{2\la}\n{x^{k}-\x}^2 + \frac{k}{\s} (F_{k}(s^k)-g_*),
\end{align}
which after multiplying by $\s$ and setting $\b = \frac{1-\la
\n{A}^2}{2\t}$ we can rewrite as
\begin{align}\label{lin_alter-32}
& \frac{1}{2\t}\n{x^{k+1}-\x}^2 +(k+1)(F_{k+1}(s^{k+1})-g_*) + \b
\n{x^{k+1}-x^k}^2 + \s(f(x^k)-f_*) \nonumber \\  \leq \ &
\frac{1}{2\t}\n{x^{k}-\x}^2 + k(F_{k}(s^k)-g_*).
\end{align}
Iterating the above, we obtain
\begin{align}
    \label{lin_alter-4}
& \frac{1}{2\t}\n{x^{k+1}-\x}^2 +
(k+1)\bigl(F_{k+1}(s^{k+1}) - g_*\bigr) +\b \sum_{i=0}^k
\n{x^{i+1}-x^i}^2 +\s \sum_{i=0}^k (f(x^i)-f_*) \\ \leq \ &
\frac{1}{2\t}\n{x^{0}-\x}^2 =\frac{D_x^2}{2\t},
\end{align}
where $D_x = \n{x^0-\x}$. It follows that
 \begin{equation}
     \label{lin_alter:5}
     F_{k}(s^{k}) - g_* \leq \frac{D_x^2}{2\t k}.
 \end{equation}

 (ii) Strong duality holds for~\eqref{lcg}. Applying the estimate obtained in \eqref{exist_dual}, we derive
\begin{align}    \label{liner-2}
\frac{1}{2\t}\n{x^{k}-\x}^2 +\s k^2(f(s^k)-f_*) -D_y k
\sqrt{2(f(s^k)-f_*)}  & + \b \sum_{i=0}^{k-1} \n{x^{i+1}-x^i}^2
                        \nonumber \\ &+\s
\sum_{i=0}^{k-1} (f(x^i)-f_*) \leq \frac{D_x^2}{2\t}.
\end{align} 
Let $t = k \sqrt{f(s^k)-f_*}$. Then from \eqref{liner-2} it
follows that $ \s t^2 - \sqrt 2 D_y t \leq D_x^2/2\t$ and therefore, we
have
\begin{equation}\label{eq:t}
    t = k \sqrt{f(s^k)-f_*} \leq \frac{D_y + \sqrt{D_y^2+ \s
    D_x^2/\t}}{\sqrt 2 \s}.
\end{equation}
By this, we show that $f(s^k)-f_* = O(1/k^2)$. Since
$t \mapsto \s t^2 - \sqrt 2 D_y t$ is bounded below by the constant
$-\frac{D_y^2}{2\s}$, from \eqref{liner-2} we conclude that $(x^k)$ is
bounded, $\n{x^{k}-x^{k-1}}\to 0$, $f(x^k)-f_* =O(1/k)$, and
\begin{equation}
    \label{eq:7_estim}
- \frac{D_y^2}{2 \s}  \leq   k(F_k(s^k) - g_* ) \leq \frac{D_x^2}{2\t}.
\end{equation}
From the last inequality we have that $|F_k(s^k)-g_*|=O(1/k)$ and since $f(s^k)-f_* =
O(1/k^2)$, we may deduce that  $|g(s^k)-g_*| = O(1/k)$.
By the definition of $z^k$, we also obtain
$f(z^k)-f_* = O(1/k^2)$.
It only remains to prove that $(x^k)$ is convergent. First we show
that all cluster points of $(x^k)$ belong to $S$. Let $(x^{k_i})$ be any
subsequence that converges to $\tilde x$. By the above, we know that
$\tilde x$ is feasible, that is $f(\tilde x)=f_*$. By  the prox-inequality,
we have
\begin{equation}\label{eq:lim}
     \lr{x^{k_i}-x^{k_i-1},\x-x^{k_i}} + k_i\la \lr{A^{\T}(Az^{k_i}-b), \x - x^{k_i}} \geq \t (g(x^{k_i})-g_*).
 \end{equation}
If we want to tend $k_i\to \infty$, we need to know how to estimate the
second term in the left-hand side of~\eqref{eq:lim}. Using that
$A^{\T}b = A^{\T}A\x$, we derive
\begin{align*}
    k_i \lr{A^{\T}(Az^{k_i}-b), \x - x^{k_i}} & = k_i \lr{A(z^{k_i} - \x),
    A(\x - x^{k_i})} \leq k_i \n{A(z^{k_i} - \x)} \cdot \n{A(x^{k_i}-\x)}\\
    & = 2k_i \sqrt{f(z^{k_i})-f_*}\sqrt{f(x^{k_i})-f_*} \to 0,
\end{align*}
due to the obtained asymptotics for $f(z^k)$ and $f(x^k)$. Hence,
passing to the limit in \eqref{eq:lim} and using that
$x^k-x^{k-1}\to 0$, we deduce $ 0 \geq \t (g(\tilde x)-g_*)$.  This
means that $\tilde x \in S$ and therefore, all cluster points of
$(x^k)$ belong to $S$. From \eqref{lin_alter-3} it follows that
\begin{equation}
    \label{bi:eq:11}
    \frac{1}{2\t}\n{x^{k+1}-\x}^2  + (k+1)(F_{k+1}(s^{k+1})-g_*) \\ \leq \frac{1}{2\t}\n{x^{k}-\x}^2 +  k(F_{k}(s^{k})-g_*).
\end{equation}
As $\bigl(k(F_k(s^k)-g_*)\bigr)_k$ is bounded below by \eqref{eq:7_estim}, we can
apply~Lemma~\ref{fejer} and conclude that the sequence $(x^k)$
converges to some element in $S$. The convergence of $(s^k)$ follows
immediately.

(iii) First, we observe that from \eqref{lin_alter:5} we have
\begin{equation}\label{fsk}
    \s (f(s^k)-f_*)\leq \frac{D_x^2}{2\t k^2} + \frac{g_* - g(s^k)}{k}.
\end{equation}
Since, $g$ is bounded below, we obtain that $f(s^k)-f_* = O(1/k)$.
Now we show that the sequence $(s^i)_{i\in \mathcal{I}}$ with
$\mathcal{I} = \{i\colon g(s^i) < g_*\}$ is bounded. To this end, we
use arguments from~\cite{solodov2007explicit}. By our assumption the
set $S = \{x\colon g(x)\leq g_*, f(x)\leq f_*\}$ is nonempty and
bounded. Consider the convex function
$\varphi(x) = \max\{g(x)- g_*, f(x)-f_* \}$. Notice that $S$ coincides
with the level set $\mathcal{L}(0)$ of $\varphi$:
\begin{equation*}
    S = \mathcal{L}(0) = \{ x\colon \varphi(x)\leq 0\}.
\end{equation*}
Since $\mathcal{L}(0)$ is bounded, $\mathcal{L}(c)=\{x\colon g(x)\leq
c\}$ is bounded for any $c\in \R$ as well. Fix any
$c\geq 0$ such that $f(s^{k})-f_*\leq c$ for all $k$. As
$g(s^{i})-g_*<0\leq c$ for $i \in \mathcal{I}$, we have that
$s^{i}\in \mathcal{L}(c)$, which is a bounded set. Hence,
$(s^{i})_{i\in \mathcal I}$ is bounded.

Now we prove the boundedness of the whole sequence $(s^k)_{k\in
\N}$. Let $M>0$ be any constant that bounds from above
$(\n{s^i})_{i\in \mathcal I}$ and $D_x + \n{\x}$.  For every index $k$ we have two
alternatives: either $g(s^k)< g_*$ or $g(s^k)\geq g_*$. If the latter
holds, then
\begin{equation*}
    \n{x^k} \leq \n{x^k-\x} +\n{\x} \leq \n{x^0-\x} + \n{\x} = D_x + \n{\bar x}\leq M,
\end{equation*}
where the second inequality holds because of \eqref{lin_alter-4}.  If
the former holds, then by the above arguments we know that
$\n{s^k}\leq M$. Assume that for the index $k$, $\n{s^k} \leq M$. If
for the index $k+1$, $g(s^{k+1})< g_*$, then we are done:
$k+1\in \mathcal{I}$ and hence, $\n{s^{k+1}}\leq M$. If
$g(s^{k+1})\geq g_*$, then $\n{x^{k+1}}\leq M$. Observe that
\begin{equation}
    \n{s^{k+1}} = \frac{\n{k s^k + x^{k+1}}}{k+1}\leq \frac{k}{k+1}M +
    \frac{1}{k+1}M = M,
\end{equation}
which completes the proof that $(s^k)$ is bounded. As
$f(s^k)-f_*=O(1/k)$, all cluster points of $(s^k)$ are feasible.
Taking the limit in \eqref{lin_alter:5} and using that $g$ is lsc, we
can also conclude that all cluster points of $(s^k)$ belong to
$S$. This guarantees that the whole sequence $(g(s^k))$ converges to
$g_*$. By this, one can improve the obtained estimate
for $f(s^k)-f_*$. In particular, now from \eqref{fsk} we have
$f(s^k)-f_* = o(1/k)$ and the proof is complete.
\end{proof}
\begin{remark}\
    \begin{enumerate}[(a)]

        \item Note that all known proofs of the PDHG algorithm
        cover only the case (ii), but even in that case they can show
        convergence only in the consistent case, \emph{i.e.,} when $f_*=0$.

        \item Let $g$ be the indicator function $\d_C $ of some closed
        convex set $C$. Then in this case \eqref{fsk} indicates that
        the performance of the primal-dual method does not depend on
        the ratio $\s/\t$ and the strong duality assumption.  This is
        natural, as now one can formulate problem~\eqref{lin_con} as a
        constrained least squares problem.

        \item Notice that one can easily make a solution set of
        problem \eqref{lcg} bounded by adding $\rho \n{x}^2$ to the objective $g$
        for some small number $\rho>0$. In this case, the solution will
        be unique $S = \{\x\}$, and hence $s^k\to \x$. Quite often it
        is possible to show that $\x$ will be not far from the actual
        solution of \eqref{lcg}.
    \end{enumerate}
\end{remark}

\subsection{Consequences} \label{sub:cons}
\paragraph{\textbf{Complexity estimates}} Consider the case when problem
\eqref{lin_con} is feasible, \emph{i.e.}, $f_* = 0$, and the strong duality
holds. For this case we will derive explicit estimates for $\e$-optimality
and compare them with existing ones.

A vector $x\in \R^n$ is called an $\e$--approximate solution of
\eqref{lin_con} if it satisfies
\begin{equation}
    \label{eq:approx}
    |g(x) - g_*| \leq \e \quad \text{and}\quad \n{Ax-b} \leq \e.
\end{equation}
Let $(x^*, y^*)$ be any saddle point of \eqref{lin_con}. As  $ \n{As^k-b} =
\sqrt{2(f(s^k)-f_*)}$, from \eqref{eq:t} we derive
\begin{equation}
    \label{feas_gap}
    \n{As^k-b} \leq \frac{D_y + \sqrt{D_y^2 +\s D_x^2 /\t}}{\s k},
  \end{equation}
  where we recall $D_x = \n{x^0 - x^*}$ and $D_y = \n{y^*}$.
Similarly, we obtain
\begin{equation}
    \label{gap_g}
    -\frac{D_y^2 + D_y\sqrt{D_y^2+\s D_x^2 /\t}}{\s k} \leq
    g(s^k)-g_* \leq \frac{D_x^2}{2\t k},
\end{equation}
where the first inequality follows from~\eqref{exist_dual}  and the
second one from~\eqref{eq:7_estim}.
From \eqref{feas_gap} and \eqref{gap_g} it is clear that we need
$O(1/\e)$ iterations to obtain $\e$-solution.  Quite remarkably, there
are several papers \cite{tran2018smooth,dvurechensky2018computational}
for solving problem~\eqref{lin_con} that also use Tseng's method,
where it is applied to the dual smoothed problem.  Moreover, we
observe that our estimate for the feasibility gap and the lower
estimate for the objective $g(s^k)-g_*$ are exactly the same as the
ones obtained in \cite[Theorem 3]{tran2018smooth}.  However, our upper
bound for the objective $g(s^k)-g_*$ is still tighter than the one in
\cite{tran2018smooth}. More specifically, Algorithm~1 proposed in
\cite{tran2018smooth} generates the sequence $(s^k)$, for which (using
our notation) it holds 
\[g(s^k)-g_* \leq \frac{D_x^2}{2\t k} + D_y  \n{As^k-b} +
    \frac{D_y^2}{\s (k+1)}. \]
Overall, this also leads to $O(1/k)$ rate, but due to additional two
terms in the above inequality, our constant is better. Since the authors
in~\cite{tran2018smooth} claim that their method achieves the
best-known rate for the non-smooth settings, we believe our estimates
improve their findings.  \smallskip

\paragraph{\textbf{Implementation details}} It is interesting to
remark that the schemes \eqref{pd2} or \eqref{eq:pd-simpler} require
even less memory than the original PDHG method. In particular, at
every moment we have to keep only two vectors $x^k, s^k\in \R^n$. In
contrast, for the PDHG algorithm, as one can see from \eqref{eq:1}, we
have to store $x^k, x^{k-1}\in \R^n$ and $y^k\in \R^m$. Moreover, in
the case $m \ll n$, it might be more efficient to switch to the dual
variables by using $\tilde x^k = Ax^k$, $\tilde s^k = As^k$. In this
notation, the scheme \eqref{eq:pd-simpler} can be rewritten as
\begin{equation}    \label{pd2-dual}
    \begin{cases}
        \tilde x^k  = Ax^k\\
        x^{k+1}  = \prox_{\t g}(x^k -\la A^{\T}(\tilde x^k + k \tilde
        s^k - b))\\
        \tilde s^{k+1} =(\tilde x^{k+1}+k \tilde s^k)/(k+1).
\end{cases}
\end{equation}
This scheme preserves the same amount of computation per iteration,
but requires us to store only one primal variable $x^k\in \R^n$ and
two dual variables $\tilde x^k, \tilde s^k \in \R^{m}$, which in the
case $2m<n$ is cheaper than the schemes \eqref{eq:1} or \eqref{pd2}
do.  Finally, in the case $m\gg n$ there is another possibility to
precompute $A^{\T}A\in \R^{n\times n}$ and use it in all iterations of
\eqref{pd2} or \eqref{eq:pd-simpler}.

\paragraph{\textbf{Connection to penalty methods}} Another approach to solve
\eqref{lin_con} or more general problem \eqref{lcg} is the
penalty method. It consists in solving a sequence of unconstrained
optimization problems
\begin{equation}\label{lin_Fk}
    \min_x g(x) + \frac{\rho_k}{2} \n{Ax-b}^2, \qquad \rho_k >0
\end{equation}
for some increasing sequence $\rho_k \to \infty $ as $k\to
\infty$. Intuitively it is clear that with larger $\rho_k$, solutions
of \eqref{lin_Fk} become closer to a solution of our constrained
problem \eqref{lcg}. For more rigorous treatment on this subject, see
\cite{nesterov2013introductory,fletcher2013practical}. In general,
penalty methods do not use duality arguments, although one still needs
them in order to obtain some convergence rates or even to prove global
convergence~\cite{fletcher2013practical}. As an exception, there is a
recent paper~\cite{necoara2015complexity} that studies the conic
optimization problem without assuming that there exists a Lagrange
multiplier. The authors applied the accelerated gradient method for
the penalized objective (although different from \eqref{lin_Fk}) and
derived $O(1/k^{\frac{2}{3}})$ estimate for the feasibility gap.

Generally speaking, we cannot solve just one problem~\eqref{lin_Fk}
for $\rho_k$ large enough. First, because we do not know which
$\rho_k$ is large enough for approximation of the true solution. And
second, because solving~\eqref{lin_Fk} in practice becomes difficult
for large $\rho_k$. Thus, penalty methods require solving a sequence
of optimization problems \eqref{lin_Fk}, which can be quite
costly. Clearly, for a specific choice $\rho_k = k \s $ problem
\eqref{lin_Fk} becomes nothing more than just $\min_x F_k(x)$.  What
we have shown is that the primal-dual method provides a nice
alternative to penalty methods. Instead of solving a sequence of
problems \eqref{lin_Fk}, it runs one iteration of something similar to
the proximal gradient method for each of the problems~\eqref{lin_Fk}
with $\rho_k = k\sigma$. In the literature this is known as a
\emph{diagonal penalty} method, see a nice overview of such methods in
\cite{garrigos2018iterative}. In general, diagonal penalty methods are
a modification of some known algorithms with an appropriate penalty
function; proving their convergence can be tricky and usually it
requires additional assumptions like strong convexity.  Thus, it is
quite remarkably that the vanilla PDHG method is unintentionally a
diagonal penalty method. We also want to note that some other recent
methods~\cite{tran2018smooth,tran2018proximal} have also this property
of a diagonal penalty method.
\smallskip

\paragraph{\textbf{Connection to inverse problems}}  A central problem in inverse
problems is solving a linear system $Ax^\dagger +\e = b$, where the
matrix $A \in \R^{m\times n}$ is given, the vector $b\in \R^m$ is
observed, and $\e\in \R^m$ is some noise. Since in most cases these
problems are ill-posed and the noise $\e$ is unknown, in order to
solve them we have to impose an appropriate regularization. The most
common approach is to consider the Tikhonov regularization:
\begin{equation}
    \label{eq:reg}
    \min_x g(x) + \frac{\gamma}{2}\n{Ax - b}^2,
\end{equation}
where $g$ is the regularizer that promotes some desirable properties
of a solution such as sparsity, smoothness, etc., and $\c >0$ is the
regularization parameter. The question of how to choose this parameter
is the main concern of such approach. In a theory the best thing would
be to solve a sequence of problems \eqref{eq:reg} with different
$(\c_k)_{k\in \N}$ and choose among solutions the best; apparently it
is not the most practical way. In our notation it means that we would
like to choose the ``best'' $\hx^k$ (whatever it means) among all
$k\in \N$ such that
\begin{equation}\label{hx^k}
\hx^k\in \argmin_x F_k(x):= g(x) + \s k (f(x)-f_*).
\end{equation}
It is clear that the problem \eqref{eq:reg} with $\c = \s k$ is
equivalent to \eqref{hx^k}.

Instead of choosing parameter $\gamma$ or solving a sequence of
problems \eqref{hx^k}, one can apply the PDHG directly to the problem
\eqref{lin_con}. Of course due to the noise, the linear system $Ax=b$
might be inconsistent, however this is not our concern, as we have
already shown that the iterates $(x^k)$ of the PDHG method will still
converge to a solution of a more general problem
\eqref{lcg}. Moreover, we can show that
$|F_k(\hx^k) - F_k(s^k)| = O(1/k)$.

Assume that the strong duality holds. The same estimation as in
Theorem~\ref{th_lin_cons} (ii) provides us
\begin{equation}\label{eq:73}
    - \frac{D_y^2}{2\s}  \leq   k(F_k(\hx^k) - g_* ).
\end{equation}
From this it follows that $g_* \leq \frac{D_y^2}{2\s k} +
F_k(\hx^k)$. Since $F_{k}(\hx^{k}) \leq F_{k}(s^{k})$, from \eqref{eq:7_estim} we have
\begin{equation}\label{eq:72}
    - \frac{D_y^2}{2\s k}  \leq F_k(\hx^k) - g_*\leq    F_k(s^k) - g_*
    \leq \frac{D_x^2}{2\t k}.
\end{equation}
 Combining the latter two inequalities, we can conclude
\begin{equation}
    \label{conclude}
    F_k(\hx^k)\leq F_k(s^k) \leq \frac{D_x^2}{2\t k} + \frac{D_y^2}{2\s
    k} + F_k(\hx^k),
\end{equation}
and hence $|F_k(s^k) - F_k(\hx^k)| = O(1/k)$. This means that by
applying the PDHG algorithm only one time for one problem
\eqref{lin_con}, we approach to each of the solutions $\hx^k$ of the
regularized problem.
\smallskip

\paragraph{\textbf{Distributed optimization}}
Here we will show that the new scheme~\eqref{eq:pd-simpler} applied to
the distributed optimization enjoys much better properties than the
original PDHG method.  Assume we have a connected simple graph
$G = (V,E)$ of $n=|V|$ computing units $v_i$, each having access to a
convex lsc function $g_i\colon \R^{d} \to (-\infty, +\infty]$. Our aim
is to find a consensus on the minimum of the aggregate objective
$g_1(x_1) +\dots + g_n(x_n)$ in a decentralized way. This problem can
be written as
\begin{equation}
    \label{eq:distrib}
    \min_{x_1,\dots, x_n} g_1(x_1)+\dots + g_n(x_n) \quad \text{s.t.}\quad x_1=\dots
    = x_n.
\end{equation}
In order to rely on the decentralized computation, we only allow
communication between adjacent nodes. The standard way to impose the
network topology is to reformulate problem \eqref{eq:distrib}
exploiting the \emph{Laplacian matrix} $L \in \R^{n\times n}$ of $G$,
which is zero everywhere, except $L_{ii} = \deg v_i$ and $L_{ij} = -1$
if $(i,j)\in E$.  Then one can formulate problem \eqref{eq:distrib} as
\begin{equation}\label{eq:distr1}
  \min_{x\in \R^{n\times d}} \quad g(x)   \quad \text{s.t.} \quad L x = 0,
\end{equation}
where $x = (x_1,\dots, x_n)^{\T} \in \R^{n\times d}$,
$g(x) = g_1(x_1) + \dots +g_n(x_n)$.  It is not difficult to
see~\cite{merris1994laplacian} that condition $Lx=0$ is equivalent to
$x_1=\dots = x_n$. Notice that multiplication of $L$ with the current
iterate $x^k$ amounts to the exchange of information between the
respective nodes.  Of course, problem~\eqref{eq:distr1}
is a particular case of \eqref{lin_con}. The PDHG~\eqref{eq:1} applied
to \eqref{eq:distr1} requires each node $v_i$ to store
$x_i^{k}, x_{i}^{k-1}$, and $y_i^k$ in each iteration. Most important
is that we have two matrix-vector multiplications $L\x^k$ and
$Ly^{k+1}$, which implies two communications per iteration.

Can scheme~\eqref{eq:pd-simpler} provide us something new? In fact,
yes. Note that instead of problem~\eqref{eq:distr1} one can consider
an equivalent problem
\begin{equation}\label{eq:distr2}
  \min_{x\in \R^{n\times d}} \quad g(x)   \quad \text{s.t.} \quad  \sqrt{L} x = 0.
\end{equation}
Since $L$ is symmetric semidefinite, $\sqrt{L}$ is well-defined and
symmetric semidefinite as well. As $\sqrt L \sqrt L = L$, the scheme~\eqref{eq:pd-simpler}
boils down to
\begin{equation}\label{eq-distr:pd-simpler}
      \begin{cases}
        x^{k+1} = \prox_{\t g}(x^k - \la L(x^k +  ks^k))\\
        s^{k+1} = (x^{k+1} + ks^k)/(k+1),
    \end{cases}
\end{equation}
which means that we need only one communication between the nodes! The
reason why we could not apply the original PDHG method \eqref{eq:1} to
problem \eqref{eq:distr2} is obvious: in general $\sqrt L x^k$ is not
related to communication of nodes: it might require exchanging of
information between nodes that are  not directly connected.

It remains to notice that the complexity estimates will be also much
better for problem~\eqref{eq:distr2} than for
\eqref{eq:distr1}. Indeed, the latter requires $\t \s \n{L}^2<1$ while
for the former we need only $\t \s \n{L}<1$. For Laplacian matrices of
graphs we know \cite{merris1994laplacian} that
$\n{L} \geq d_{\max}+1>1$, where $d_{\max}$ denotes the largest degree
of the nodes of $G$.  If we assume that the strong duality holds for
\eqref{eq:distr1} and a $(x^*, u^*)$ is a saddle point, then it also
holds for \eqref{eq:distr2} and $(x^*, \sqrt L u^*)$ would be a
respective saddle point. Thus, the bound for $D_y$ in
\eqref{exist_dual} for \eqref{eq:distr2} is also not bigger than the
one for \eqref{eq:distr1}. Hence, when the strong duality holds, from
\eqref{feas_gap} and \eqref{gap_g} we can conclude that the complexity
for \eqref{eq:distr2} is better than the one for \eqref{eq:distr1}.

For instance, in~\cite{lan2017communication} the authors develop a
PDHG-based algorithm for decentralized communication (for the case when
$\prox_g$ is not easy to compute). However, they use a standard form
of the PDHG and because of that each iteration of their algorithm
needs two communications per iteration.

The idea to use constraints $\sqrt L x = 0$ in \eqref{eq:distr2} is
quite standard, see for example
\cite{uribe2017optimal,jakovetic2015linear,pmlr-v70-scaman17a}. However such algorithms
require either more restrictive assumptions (strong convexity,
dual-friendliness, etc.) or they have more expensive iterations. Also,
using the results from \cite{luke2018block}, one can derive a
stochastic extension of \eqref{eq-distr:pd-simpler} where in every
iteration only a small random subset of nodes communicate between
themselves. 
As a remark, we note that the above discussion will be
also valid if instead of $L$ one considers a more general weighted
Laplacian matrix.

\section{Generalization} \label{sec:gen}
\subsection{Additional smooth term}
Assume now that we are given more structure in  problem \eqref{lin_con}:
\begin{equation}\label{lin_con2}
  \min_{x\in \R^{n}} \quad g(x) + h(x)   \quad \text{s.t.} \quad Ax = b,
\end{equation}
where in addition to the previous settings we assume that
$h\colon \R^n\to \R$ is a convex differentiable function with
$\beta$--Lipschitz gradient, that is
\[\n{\nabla h(u) - \nabla h(v)}\leq \beta \n{u - v}\qquad
    \forall u, v\in \R^n.\] In most cases computing the $\prox_{g +h}$
is not practical anymore, thus the vanilla PDHG method will not be
efficient for this problem. Condat and V\~u in
\cite{Condat2013,vu2013splitting} proposed an extension of the PDHG
algorithm to deal with such cases. Applied to \eqref{lin_con2}, this
algorithm is given by
\begin{equation}\label{eq:12}
    \begin{cases}
        y^{k+1}  =  y^k + \sigma (A \x^k -b)\\
        x^{k+1}  = \prox_{\t g}(x^k - \tau (A^{\T} y^{k+1} + \nabla h(x^k))).
    \end{cases}
\end{equation} 
Its convergence can be proved under the assumptions that the solution
set is nonempty, the strong duality holds and $\t \s \n{A}^2 < 1
-\t \beta$.

It is clear that in the same way as in section~\ref{sec:lin:appr} one
can transform \eqref{eq:12} into entirely primal algorithm:
\begin{equation}    \label{pd22}
\begin{cases}
    z^k  = (x^k+ks^{k})/(k+1)\\
    x^{k+1}  = \prox_{\t g}(x^k -\t \nabla h(x^k)- (k+1)\la \nabla f(z^k))\\
     s^{k+1} = (x^{k+1}+ks^k)/(k+1).
\end{cases}
\end{equation}
We will not repeat the proof, rather give a key ingredient. To this
end, we first recall the inequality from the descent
lemma~\cite{nesterov2013introductory} for $\b$--smooth function $h$:
\begin{equation}
    \label{desc_lemma}
    h(u)-h(v) - \lr{\nabla h(v), u-v} \leq \frac \beta 2 \n{u-v}^2
    \quad \forall u,v\in \R^n.
\end{equation}
The key ingredient is  the following estimation
\begin{align}\label{ineq_h}
  \lr{\nabla h(x^k), \x-x^{k+1}} & = \lr{\nabla h(x^k), \x-x^{k}} + \lr{\nabla h(x^k), x^k-x^{k+1}} \\ & \leq
                                  [h(\x)-h(x^k)]  + [h(x^k) - h(x^{k+1})
                                  +\frac{\beta}{2}\n{x^{k+1}-x^k}^2] \\ & =                                   h(\x) - h(x^{k+1})
                                  +\frac{\beta}{2}\n{x^{k+1}-x^k}^2 
\\ & \leq k\bigl(h(s^k)-h(\x)\bigr) - (k+1)\bigl(h(s^{k+1})-h(\x)\bigr) + \frac{\beta}{2}\n{x^{k+1}-x^k}^2,
\end{align}
where the second line follows from convexity of $h$ and decsent
inequality~\eqref{desc_lemma} and the fourth one follows from
convexity of $h$ and the definition of $s^{k+1}$.  Combining this
inequality with the similar ones as in the proof of
Theorem~\ref{th_lin_cons} we can show convergence
of~\eqref{pd22}. Evidently, in the same way we can show that
algorithm~\eqref{pd22}, and hence,~\eqref{eq:12}, in fact solves a
more general problem
\begin{equation}\label{lin_con_gen2}
  \min_{x\in \R^{n}} \quad g(x) + h(x)   \quad \text{s.t.} \quad x \in
  \argmin f,
\end{equation}
where recall that $f(x) = \frac 1 2 \n{Ax-b}^2$.

\subsection{$g$ is strongly convex}
\label{sec:strong}
When $g$ is $\c$--strongly convex, we can obtain even better
convergence rates. Although our results presented below will be valid
for the general case as in~\eqref{lin_con_gen2}, for the clarity of
presentation we consider the case when $h\equiv 0$. Hence, now our
problem reads as:
\begin{equation}\label{strong}
    \min_{x\in \R^n} \quad   g(x) \quad   \text{s.t.} \quad x\in \argmin f,
\end{equation}
where $g$ is $1$--strongly convex function that we assume without loss
of generality.

First, let us consider the case when the linear system $Ax=b$ is
consistent and strong duality holds. In this case, one can apply the
accelerated PDHG method \cite{chambolle2011first}:
\begin{equation}\label{acc-pd}
    \begin{cases}
        y^{k+1}  =  y^k + \sigma_k (A \x^k -b)\\
        x^{k+1}  = \prox_{\tau_k g}(x^k - \tau_k A^{\T} y^{k+1}),
    \end{cases}
\end{equation}
where $\x^k = x^k + \th_k (x^k-x^{k-1})$, $\th_k =
\frac{\t_k}{\t_{k-1}}$ and
\begin{equation}    \label{eq:param}
\t_{k} = \frac{\t_{k-1}}{\sqrt{1+\t_{k-1}}}, \quad  \t_k \s_k = \la,
\quad \la \n{A}^2\leq 1 \qquad \forall k\geq 0.
\end{equation}
Iterating the first equation in \eqref{acc-pd}, one can derive
\begin{align*}
       y^{k+1} & = y^k + \s_k (A\x^k-b) = y^{k-1} + A(\s_k\x^k+\s_{k-1}\x^{k-1})
       -(\s_k+\s_{k-1})b = \dots \\ &  = y^0 + A(\s_k\x^k+\dots +
       \s_0\x^0) - (\s_k + \dots + \s_0)b.
\end{align*}
Similarly as in section~\ref{sec:lin:appr}, one may introduce $z^k$,
defined by $z^0 = x^0$, $z^1 = x^1$ and
\begin{equation*}
    z^k = \frac{\s_k \x^k+\dots +\s_0 \x^0}{\s_k+\dots + \s_0} =
    \frac{(\s_k+\s_{k-1}) x^k+ \s_{k-2}x^{k-1}+\dots +\s_0 x^1}{\s_k+\dots + \s_0},
\end{equation*}
for $k\geq 2$. Here we have used that
$\x^k = x^k + \th_k (x^k-x^{k-1}) = x^k + \frac{\s_{k-1}}{\s_k}
(x^k-x^{k-1}) $. Let $\Sigma_k:= \s_k+\dots + \s_0$ for $k\geq 0$ and
$\Sigma_{-1} = 0$. Define sequence $(s^k)$ as $s^0=x^0$ and
$s^{k} = \frac{\s_{k-1}x^k + \dots + \s_0 x^1}{\Sigma_{k-1}}$. For
simplicity, we again assume that $y^0 = 0$, thus
$y^{k+1} = \Sigma_{k}(Az^k-b)$ for all $k\geq 0$. Then the primal-dual
scheme~\eqref{acc-pd} might be written in the primal form:
\begin{equation} \label{lin_cons:accel}
    \begin{cases}
    z^k  = (\s_k x^k  + \Sigma_{k-1} s^k)/\Sigma_k\\
    x^{k+1} = \prox_{\t_k g}(x^k - \t_k \Sigma_k \nabla f(z^k))\\
    s^{k+1} = (\s_k x^{k+1}+ \Sigma_{k-1}s^k)/\Sigma_k,
\end{cases}
\end{equation}
where $k\geq 0$, $\t_k, \s_k$ satisfy \eqref{eq:param} and $\Sigma_k = \s_k+\dots+\s_0$. We show that this algorithm in fact solves \eqref{strong}.

Let $g_*$ be the optimal value of \eqref{strong},
$F_k(x) = g(x) + \Sigma_k(f(x)-f_*)$ be the penalty function, and
$\hx^k$ be the unique minimizer of $F_k$.
\begin{theorem}\label{th:g-strong}
    Let $(x^k)$, $(s^k)$ be generated by
    \eqref{lin_cons:accel}, $\la \n{A}^2 \leq 1$, and the solution set
    $S = \{\x\}$. Then it holds
    
    \begin{enumerate}[(i)]
        \item $(s^k)$ converges to $\x$,
        $F_k(\hat x^k) - g_* \leq F_k(s^k) - g_* =O(1/k^2)$,
        $f(s^k)-f_* = o(1/k^2)$.
        \item If strong duality holds for problem~\eqref{strong}, then
        $(x^k)$ also converges to $\x$ at the rate
        $\n{x^k-\x} = O(1/k)$ and  $f(x^k)-f_*=O(1/k^3)$, 
        $f(s^k)-f_* = O(1/k^4)$, $|F_k(s^k)-g_*|=O(1/k^2)$,
        $|g(s^k)-g_*| = O(1/k^2)$.
    \end{enumerate}  
\end{theorem}
\begin{proof}
 By the prox-inequality~\eqref{prox_charact-strong},
\begin{multline}\label{str:0}
   \frac{1}{\t_k} \lr{x^{k+1}-x^k, \x-x^{k+1}} +  \s_k \lr{\nabla f(x^k),\x-x^{k+1}} +
 \Sigma_{k-1} \lr{\nabla f(s^k), \x-x^{k+1}} \\ \geq g(x^{k+1}) -g(\x)
+ \frac 1 2 \n{x^{k+1}-\x}^2.
\end{multline}
Using identities~\eqref{identities}, we obtain
\begin{multline}\label{str:lin_alter-1}
\Sigma_k(f(x^{k+1})-f_*)  +\s_k (f(x^k)-f_*) +
\Sigma_{k-1}(f(s^k)-f_*) \\+ (g(x^{k+1}) -g_*) + \frac 1 2
\n{x^{k+1}-\x}^2  -\frac{\s_k}{2}\n{A(x^{k+1}-x^k)}^2  \\ \leq \frac{1}{\t_k}\lr{x^{k+1}-x^k, \x-x^{k+1}} +\frac{\Sigma_{k-1}}{2}\n{A(x^{k+1}-s^k)}^2.
\end{multline}
Convexity of $F_{k+1}(x)= g(x)+ \Sigma_k (f(x)-f_*)$ and the
property~\eqref{eq:quad-id1} for $f$ yield
\begin{equation}\label{str:1}
\s_k F_{k+1}(x^{k+1}) + \Sigma_{k-1} F_{k+1}(s^k) \geq \Sigma_k F_{k+1}(s^{k+1})
+\frac{\s_k \Sigma_{k-1}}{2}\n{A(x^{k+1}-s^k)}^2.
\end{equation}
Applying \eqref{str:1} to \eqref{str:lin_alter-1} and using that
$\Sigma_{k-1} F_{k+1}(s^k)= \Sigma_{k-1} F_k(s^k) +  \s_k\Sigma_{k-1}(f(s^k)-f_*)$, we obtain
\begin{align}\label{lin_alter-22}
\frac{\Sigma_k}{\s_k}F_{k+1}(s^{k+1})-g_*) + \s_k (f(x^k)-f_*)
& -\frac{\s_k}{2}\n{A(x^{k+1}-x^k)}^2 + \frac 1 2 \n{x^{k+1}-\x}^2 \nonumber \\ &\leq \frac{1}{\t_k}\lr{x^{k+1}-x^k,
\x-x^{k+1}}+ \frac{\Sigma_{k-1}}{\s_k}(F_k(s^k)-g_*).
\end{align}
From the Cosine Law \eqref{eq:cosine} and
$\n{A(x^{k+1}-x^k)}\leq \n{A} \n{x^{k+1}-x^k}$ it follows that
\begin{multline}\label{str:lin_alter-3}
\frac{1+\t_k}{2\t_k}\n{x^{k+1}-x}^2 +\frac{\Sigma_k}{\s_k}(F_{k+1}(s^{k+1})-g_*) +
\frac{1-\la \n{A}^2}{2\t_k} \n{x^{k+1}-x^k}^2 \\ + \s_k(f(x^k)-f_*)  \leq
\frac{1}{2\t_k}\n{x^{k}-x}^2 + \frac{\Sigma_{k-1}}{\s_k} (F_{k}(s^k)-g_*).
\end{multline}
By the definition of $(\t_k)$, we have $\frac{\t_{k}\s_{k+1}}{\s_k
\t_{k+1}} = \frac{\t_k^2}{\t_{k+1}^2}=(1+\t_k)$ and $\t_k \s_k =
\la$. Multiplying the left and right hand-sides of \eqref{str:lin_alter-3}
by $\s_k$ and using the latter identities, we deduce
\begin{equation}
    \label{lin_cons:accel-2}
    \frac{\s_{k+1}}{2\t_{k+1}}\n{x^{k+1}-\x}^2 +
\Sigma_k (F_{k+1}(s^{k+1}) - g_*)+
\s_k^2 (f(x^k)-f_*) \leq     \frac{\s_k}{2\t_{k}}\n{x^{k}-\x}^2 + \Sigma_{k-1} (F_{k}(s^{k}) - g_*).
\end{equation}
Iterating the above and recalling that $\Sigma_{-1} = 0$, we obtain
\begin{equation}
    \label{lin_cons:accel-3}
    \frac{\s_{k}}{2\t_{k}} \n{x^{k}-\x}^2 +
    \Sigma_{k-1}(F_{k}(s^{k}) - g_*) +
    \sum_{i=0}^{k-1}\s_i^2(f(x^i)-f_*) \leq
     \frac{\s_0}{2\t_{0}}\n{x^{0}-\x}^2 =   \frac{\s_0}{2\t_{0}}D_x^2.
 \end{equation}
 For simplicity, assume that $\t_0 = 1$. Then it is not difficult to
prove by induction that $\frac{2}{k+2} \leq \t_k \leq
\frac{3}{k+2}$. In the general case, all results will be the same up
to some constants, as it is known from \cite{chambolle2011first} that
$\t_k\sim 1/k$. In the case $\t_0=1$, we have
$\frac{k+2}{2}\la \geq \s_k \geq \frac{k+2}{3}\la$, hence
\begin{equation*}
\frac{(k+4)(k+1)}{4}\la \geq    \Sigma_k \geq \frac{(k+4)(k+1)}{6}\la.
\end{equation*}
From \eqref{lin_cons:accel-3} it follows that
$\Sigma_k (F_{k+1}(s^{k+1}) - g_*) \leq \frac{\la}{2}D_x^2 $, where we
took $\s_0 = \la $, due to $\t_0 = 1$. Thus,
 \begin{equation} \label{lin_str:k2}
F_{k+1}(s^{k+1})
 - g_* \leq \frac{\la D_x^2}{2 \Sigma_k}  \leq
 \frac{3D_x^2}{(k+4)(k+1)} = O(1/k^2).
\end{equation}
As $g$ is strongly convex, it is bounded below. The set
$S = \{\x\}$ is of course bounded, thus we can use the same arguments
as in part (iii) of Theorem~\ref{th_lin_cons} to conclude that $(s^k)$
is bounded, $s^k \to \x$ and $g(s^k) \to
g_*$. Equation~\eqref{lin_str:k2} also yields
$f(s^{k+1}) \leq \frac{\la D_x^2}{2 \Sigma_k^2} +
\frac{g_*-g(s^k)}{\Sigma_k} = o(1/k^2)$, since $g(s^k) \to g_*$ and
$1/\Sigma_k = O(1/k^2)$.
\smallskip

Case (ii). Strong duality holds for~\eqref{strong}.  Using that
\(g(s^k)-g_*\geq - D_y \cdot \sqrt{2(f(s^k)-f_*)}\), which is a
consequence of~\eqref{exist_dual}, one has
from~\eqref{lin_cons:accel-3}
\begin{equation}\label{lin_cons:accel-bound}
\frac{\s_k}{2 \t_k}\n{x^{k}-\x}^2 +\Sigma_{k-1}^2 (f(s^k)-f_*)  - D_y \Sigma_{k-1}
\sqrt{2(f(s^k)-f_*)} +  \sum_{i=0}^{k-1} \s_i^2 (f(x^i)-f_*) \leq
\frac{\la D_x^2}{2}.
\end{equation}

Let $t = \Sigma_{k-1} \sqrt{f(s^k)-f_*}$. Then from
the last equation it follows that
$ t^2 - \sqrt 2 D_y t \leq \frac{\la D_x^2}{2}$, from which one can
derive that $t \leq \frac{D_y + \sqrt{D_y^2+ \la D_x^2}}{\sqrt 2}$. By
this, we show that $f(s^k)-f_* = O(1/k^4)$. Since
$ t^2 -\sqrt 2 D_y t$ is bounded below by the constant
$-\frac{D_y^2}{2}$ and $\frac{\t_k}{\s_k}\leq \frac{(k+2)^2}{9\la}$, we conclude that
\begin{equation*}
    \n{x^k - \x}^2 \leq \frac{\t_k}{\s_k} (\la D_x^2 + D_y^2) \leq
    (D_x^2 + \frac{D_y^2}{\la}) \frac{9}{(k+2)^2}.
\end{equation*}
Recall that $t^2-\sqrt 2 D_y t = \Sigma_{k-1}(F_k (s^k)-g_*)$, thus
\begin{equation}
    \label{eq:7}
    - \frac{D_y^2}{2}  \leq   \Sigma_{k-1}(F_k(s^k) - g_* ) \leq \frac{\la
    D_x^2}{2}.
\end{equation}
From this we observe that $|F_k(s^k)-g_*|=O(1/k^2)$ and due to the
asymptotic of $f(s^k)-f_*$ we have $|g(s^k)-g_*| = O(1/k^2)$.
Finally, from \eqref{lin_cons:accel-bound} the sequence
$\bigl(\s_i^2 (f(x^i)-f_*)\bigr)_i$ is summable and since $\s^2_k\sim k^2$,
we have that $f(x^k)-f_* = O(1/k^3)$.  \
\end{proof}

\begin{remark}
    We note that for both cases mentioned in this section it is
    straightforward to derive similar results as in
    Section~\ref{sub:cons}. Notice also that for a more general case
    $\gamma \neq 1$,  all rates in Theorem~\ref{th:g-strong}  remain the same.
\end{remark}

\bibliographystyle{acm}
\small
\bibliography{biblio}

\end{document}